\documentclass[11pt]{amsart}

  \def\F{\mathbb F}
  
  \def\a{\alpha}

  \def\e{\epsilon}

  \def\la{\langle}
  \def\ra{\rangle}

  \def\no{\noindent}

\usepackage{color}

\theoremstyle{plain}
	\newtheorem{theorem}{Theorem}[section]
	
	\newtheorem{lemma}[theorem]{Lemma}

\theoremstyle{definition}

\numberwithin{equation}{section}

\begin{document}

\title{Chiral polyhedra and finite simple groups}
\author{Dimitri Leemans}
\address{Dimitri Leemans, University of Auckland,
Department of Mathematics,
Private Bag 92019,
Auckland, New Zealand}
\email{d.leemans@auckland.ac.nz}

\author{Martin W. Liebeck}
\address{Martin W. Liebeck, Department of Mathematics, Imperial College,
London SW7 2AZ, UK}
\email{m.liebeck@imperial.ac.uk}

\date{\today}

\begin{abstract}
We prove that every finite non-abelian simple group acts as the automorphism group of a chiral polyhedron, apart from the groups $PSL_2(q)$, $PSL_3(q)$, $PSU_3(q)$ and $A_7$.
\end{abstract}
\maketitle


\section{Introduction}\label{intro}
Polyhedra and their generalisations to higher ranks, polytopes, are certain ranked partially ordered sets generalising geometric objects that have been studied since the Greeks (see~\cite[Chapter 1]{MS2002}).
Those polytopes whose automorphism group acts transitively on maximal flags are called regular. They have maximum possible rotational and reflectional symmetries. Those that are chiral have maximum rotational symmetries but no reflections. It is already known which finite simple groups are automorphism groups of abstract regular polyhedra (see below for more details).
The main purpose of this article is to determine which finite simple groups are automorphism groups of chiral polyhedra.

In order to state the main results we require precise definitions of the above terms.
Following~\cite{MS2002}, an (abstract) {\em polytope} $({\mathcal P},\leq)$ {\em of rank $n$} is a partially ordered set with a rank function ranging from $-1$ to $n$ and satisfying the following properties. The elements of rank $i$ are called the $i$-faces of $\mathcal P$.
There exists a unique least face $F_{-1}$ and a unique greatest face $F_n$. The {\em flags} are the maximal totally ordered subsets of $\mathcal P$ and they must all contain exactly $n+2$ faces.
Two flags $\Phi$ and $\Psi$ are called {\em adjacent} if they differ in exactly one face. They are called $i$-adjacent  if this face is an $i$-face.
The poset $\mathcal P$ must be {\em strongly connected}, that is, every pair of flags must be connected by a path of adjacent flags in $\mathcal P$. Finally we require that for any $(i-1)$-face $F$ and any $(i+1)$- face $G$ of $\mathcal P$ such that $F\leq G$, there are exactly two $i$-faces between $F$ and $G$.
If the rank of $\mathcal P$ is 3, we call $\mathcal P$ a {\em polyhedron}.

An automorphism of $({\mathcal P},\leq)$ is a bijection of the faces of $\mathcal P$ that preserves the order $\leq$. The set of all automorphisms of $({\mathcal P},\leq)$ with composition forms a group called the {\em automorphism group} of $({\mathcal P},\leq)$ and denoted $\Gamma({\mathcal P})$.
If $\Gamma({\mathcal P})$ has a unique orbit on the flags of ${\mathcal P}$, we say that $\mathcal P$ is {\em regular}. If $\Gamma({\mathcal P})$ has two orbits such that any two adjacent flags belong to distinct orbits, we say that $\mathcal P$ is {\em chiral}.

As defined for instance in~\cite{MS2002}, a C-{\it group} is a group $G$ generated by pairwise distinct involutions $\rho_0,\ldots,\rho_{n-1}$ which satisfy the following {\it intersection property}:
\[
\forall J, K \subseteq \{0,\ldots,n-1\}, \langle \rho_j \mid j \in J\rangle \cap \langle \rho_k \mid k \in K\rangle = \langle \rho_j \mid j \in J\cap K\rangle.
\]
A C-group $(G, \{\rho_0,\ldots, \rho_{n-1}\})$ is a {\it string} C-group if its generators satisfy the following relations:
\[ 
(\rho_j\rho_k)^2 = 1\; \forall j, k \in \{0,\ldots n-1\} \hbox{ with } \mid j-k\mid \geq 2.
\]
In~\cite{MS2002} it is shown that string C-groups and abstract regular polytopes are in one-to-one correspondence. Every string C-group gives an abstract regular polytope and, given an abstract regular polytope and one of its base flags, one can construct a set of distinguished generators that, together with the automorphism group of the polytope, give a string C-group. In particular, the automorphism group of an abstract regular polyhedron is generated by three involutions
$\rho_0,\rho_1,\rho_2$, two of which commute (namely, $\rho_0,\rho_2$). 
In 1980, it was asked in the Kourovka Notebook (Problem 7.30) which finite simple groups have this property. 
This was solved by Nuzhin and others in \cite{n1,n2,n3,n4,Mazurov03}: every non-abelian finite simple group can be generated by three involutions, two of which commute, with the following exceptions:
\[
\begin{array}{l}
PSL_3(q),\,PSU_3(q),\,PSL_4(2^n), \,PSU_4(2^{n}),\\
A_6,\,A_7,\,M_{11},\, M_{22}, \,M_{23}, \,McL.
\end{array}
\]
Thus every finite simple group, apart from the above exceptions, is the automorphism group of an abstract regular polyhedron.

Similarly, in~\cite{SW1991}, it is shown that for a finite group $G$, the chiral polyhedra having $G$ as automorphism group are in bijective correspondence with pairs $x,t \in G$ satisfying the following conditions:
\begin{itemize}
\item[{\rm (i)}] $G = \la x,t\ra$;
\item[{\rm (ii)}] $t$ is an involution;
\item[{\rm (iii)}] there is no involution $\a \in {\rm Aut}(G)$ such that $x^\a = x^{-1}$, $t^\a = t$.
\end{itemize}
Our main result classifies those finite simple groups $G$ possessing such generators:

\begin{theorem}\label{main1}
Let $G$ be a non-abelian finite simple group, not $A_7$, $PSL_2(q)$, $PSL_3(q)$ or $PSU_3(q)$. Then there exist $x,t \in G$ such that the following hold:
\begin{itemize}
\item[{\rm (i)}] $G = \la x,t\ra$;
\item[{\rm (ii)}] $t$ is an involution;
\item[{\rm (iii)}] there is no involution $\a \in {\rm Aut}(G)$ such that $x^\a = x^{-1}$, $t^\a = t$.
\end{itemize}
\end{theorem}

As a consequence, for every nonabelian finite simple group $G$ except $A_7$, $PSL_2(q)$, $PSL_3(q)$ or $PSU_3(q)$, there exists an abstract chiral polyhedron having automorphism group $G$.
This result was already known to be true for the Suzuki groups~\cite{JonesSilver93, HL2014}, the small Ree groups~\cite{Jones94}, the alternating groups (see Section~\ref{alt} for details) and some small sporadic groups~\cite{HHL2012}.

We now discuss the exceptions in Theorem \ref{main1}. 
The groups $PSL_2(q)$ do not have pairs of elements $x$, $t$ satisfying (i)--(iii) of the theorem; this a consequence of a result of Macbeath~\cite{Macbeath},  as observed by Singerman~\cite[Theorem 3]{Singerman}. 
For the group $A_7$, an exhaustive computer search shows that no pair of elements $x$, $t$ satisfying (i)--(iii) of Theorem~\ref{main1} exists -- see ~\cite{HHL2012}.

Thus it remains to consider the groups $PSL_3(q)$ and $PSU_3(q)$; it will be shown that these also do not possess generators $x,t$  satisfying conditions (i)--(iii) in a forthcoming paper~\cite{LL2016b}.

The rest of the paper is devoted to the proof of Theorem \ref{main1}.
This is divided into four cases: namely,  
the case where $G$ is of exceptional Lie type (Section~\ref{excep}), 
the case where $G$ is classical (Section~\ref{class}), 
the case where $G$ is an alternating group (Section~\ref{alt}), and the case where 
$G$ is a sporadic group (Section~\ref{sporadic}).

Observe that chiral polyhedra (and regular polyhedra) may be seen also as regular maps.
In a recent paper~\cite{JonesNew}, Gareth Jones has studied much further the link between automorphism groups of edge-transitive maps and finite simple groups.

\section{Preliminaries}

In this section we prove some lemmas needed for the proof of Theorem~\ref{main1},
The first two lemmas are straightforward hence we omit their proofs.

\begin{lemma}\label{inv}
Let $G$ be a finite group, and suppose $x,t\in G$ satisfy $x^t = x^{-1}$. Then the set $\{y\in G : x^y = x^{-1}\}$ is contained in the coset $C_G(x)t$.
\end{lemma}

\begin{lemma}\label{maxconj}
Let $G$ be a finite group with a subgroup $A$. Suppose $M$ is a maximal subgroup of $G$ containing $A$, such that any two $G$-conjugates of $A$ that are contained in $M$ are $M$-conjugate. Then the number of $G$-conjugates of $M$ containing $A$ is $|N_G(A):N_M(A)|$.
\end{lemma}

For a finite group $G$ and a positive integer $r$, denote by $I_r(G)$ the set of elements of order $r$ in $G$, and let $i_r(G) = |I_r(G)|$.

The next two lemmas provide upper and lower bounds for the numbers of involutions $i_2(G)$ in groups of Lie type.
Recall that a simple group of Lie type over a finite field $\F_q$ can be written as $(\bar G^F)'$, the derived group of the fixed point group of a Frobenius endomorphism $F$ of the corresponding simple adjoint algebraic group $\bar G$ over 
$\bar\F_q$.  

We use the notation $L_n^\e(q)$ (where $\e = \pm$) to denote $PSL_n(q)$ when $\e=+$, and $PSU_n(q)$ when $\e=-$. Similarly, $E_6^\e(q)$ denotes $E_6(q)$ when $\e=+$ and $^2\!E_6(q)$ when $\e=-$.

\begin{lemma}\label{upperinv}
Let $G$ be a finite simple group of Lie type over $\F_q$, and write $G = (\bar G^F)'$ as above. Let $r$ be the rank of $\bar G$, and $N$ the number of positive roots in the root system of $\bar G$. Define
\[
M_G = \left\{\begin{array}{l} \frac{1}{2}(N+r), \hbox{ if }G \hbox{ is of type }^2\!F_4,\, ^2\!G_2 \hbox{ or }^2\!B_2, \\
N+r,\hbox{ otherwise.}\end{array} \right.
\]
Then $i_2({\rm Aut}(G)) < 2\left(q^{M_G}+q^{M_G-1}\right)$.
\end{lemma}

\begin{proof}This is \cite[Prop. 1.3]{LLS}.
\end{proof}

\begin{lemma}\label{lowerinv}
Let $G$ be a finite simple group of Lie type over $\F_q$.
\begin{itemize}
\item[{\rm (i)}] If $G$ is of classical type, then $i_2(G) > \frac{1}{4}q^{N_G}$, where $N_G$ is defined as follows:
\[
\begin{array}{|l|l|}
\hline
G & N_G \\
\hline
L_n^\e(q)\,(\e = \pm) & [\frac{1}{2}n^2] \\
PSp_{2m}(q),\,P\Omega_{2m+1}(q) & m^2+m \\
P\Omega_{2m}^\e(q)\,(\e = \pm) & m^2-1 \\
\hline
\end{array}
\]
\item[{\rm (ii)}]  If $G$ is of exceptional type, then with one exception $i_2(G) > \frac{1}{2}q^{N_G}$, where $N_G$ is defined below; the exception is $G=E_7(q)$, in which case $i_2(G) > \frac{1}{4}q^{N_G}$.
\[
\begin{array}{|l|lllllllll|}
\hline 
G & E_8(q) & E_7(q) & E_6^\e(q) & F_4(q) & G_2(q) & ^2\!F_4(q) & ^2\!G_2(q) & ^2\!B_2(q) & ^3\!D_4(q) \\
\hline
N_G & 128 & 70 & 40& 28& 8 & 14 & 4& 3& 16 \\
\hline
\end{array}
\]
\end{itemize}
\end{lemma}

\begin{proof}
 An asymptotic version of this result is proved in \cite[Props. 4.1, 4.3]{LS2}. The lemma is easily verified by keeping track of the constants in that proof.
 \end{proof}

\vspace{4mm} Note that the exponents $M_G$ and $N_G$ in Lemmas \ref{upperinv} and \ref{lowerinv} are equal, except in cases where $G$ possesses an involutory graph automorphism (types $L_n^\e, P\Omega_{2m}^\e, E_6^\e$).

\section{Proof of Theorem \ref{main1} for $G$ of exceptional type}\label{excep}

Let $G$ be a finite simple group of exceptional Lie type over $\F_q$ (i.e. of type 
$E_8$, $E_7$, $ E_6^\e$, $F_4$, $G_2$, $^2\!F_4$, $^2\!G_2$, $^2\!B_2$ or $^3\!D_4$). 
Assume that $q>2$ when $G$ is of type $G_2$ or $^2\!F_4$, and $q>3$ for type $^2\!G_2$ (for the excluded groups, $G_2(2)' \cong U_3(3)$ and $^2\!G_2(3)'\cong L_2(8)$ and these will be dealt with as classical groups in the next section).

Write $G = (\bar G^F)'$, the derived group of the fixed point group of a Frobenius endomorphism $F$ of the corresponding simple adjoint algebraic group $\bar G$ over $\bar \F_q$. Let $d = |\bar G^F:G|$, so that 
\[
d = \left\{\begin{array}{l}
(2,q-1), \hbox{ if }G=E_7(q), \\
(3,q-\e), \hbox{ if }G=E_6^\e(q), \\
1,\hbox{ otherwise.}
\end{array} \right.
\]

By \cite[Section 2]{KS}, there is a cyclic maximal torus $\la x \ra$ of $G$ of order as given in Table \ref{tb1}. In the table, $\Phi_n(q)$ denotes the $n^{th}$ cyclotomic polynomial evaluated at $q$. Morover, $T = C_{\bar G^F}(x)$ is a maximal torus of $\bar G^F$ of order $d|\la x\ra|$, and also $C_{Aut(G)}(x) = T$ (see \cite[2.8(iii)]{seitz}).

Let $w_0$ be the longest element of the Weyl group $W(\bar G)$. Suppose that $w_0=-1$ (i.e. $\bar G \ne E_6$). Then the involutions in $N_{\bar G}(T)$ that invert $x$ are conjugates of $t_0$, a preimage in $N(T)$ of $w_0$. Now 
$C_{\bar G}(t_0)$ has dimension equal to the number of positive roots in the root system of $\bar G$. Hence we see from the lists of possibilities for involution centralizers in $G$, in \cite[Section 4.5]{GLS} for $q$ odd, and in \cite{AS} for $q$ even, that $C_G(t_0)$ is as given in Table \ref{tb1}. 

Similarly, if $\bar G = E_6$, the involutions in ${\rm Aut}(G)$ inverting $x$ are conjugate to $t_0$, a preimage in $N_{G\la \tau \ra}(T)$ of $w_0\tau = -1$, where $\tau$ is an involutory graph automorphism of $G$, and again $C_G(t_0)$ is as in Table \ref{tb1}.

\begin{table}[htb]
\caption{Torus $\la x\ra$ and centralizer $C_G(t_0)$, $G$ exceptional} \label{tb1}
\[
\begin{array}{|c|c|c|}
\hline
G & |\la x \ra | & C_G(t_0) \\
\hline
E_8(q) & \Phi_{30}(q) & D_8(q),\,q\hbox{ odd} \\
            &                         & [q^{84}].B_4(q), \,q\hbox{ even} \\
\hline
E_7(q) & \frac{1}{d}\Phi_{18}(q)\Phi_2(q),\,q\ge 3 & A_7^\e(q).2,\, q\equiv \e\hbox{ mod }4 \\
            & 129,\,q=2  & [q^{42}].B_3(q),\,q \hbox{ even} \\
\hline
E_6^\e(q) & \frac{1}{d}\Phi_{9}(q),\,\e=+ & C_4(q),\, q\hbox{ odd} \\
            & \frac{1}{d}\Phi_{18}(q),\,\e=- & [q^{15}].C_3(q),\,q \hbox{ even} \\
\hline
F_4(q) & \Phi_{12}(q),\,q\ge 3 & (A_1(q)C_3(q)).2,\, q\hbox{ odd} \\
            & 17,\,q=2  & [q^{18}].A_1(q)^2,\,q \hbox{ even} \\
\hline
G_2(q) & \Phi_{6}(q),\,q\ge 4 & (A_1(q)A_1(q)).2,\, q\hbox{ odd} \\
            & 13,\,q=3  & [q^{3}].A_1(q),\,q \hbox{ even} \\
\hline
^3\!D_4(q) & \Phi_{12}(q) & (A_1(q)A_1(q^3)).2,\, q\hbox{ odd} \\
            &   & [q^{9}].A_1(q),\,q \hbox{ even} \\
\hline
^2\!F_4(q) & q^2+\sqrt{2q^3}+q+\sqrt{2q}+1 & [q^{9}].A_1(q) \\
\hline
^2\!G_2(q) & q+\sqrt{3q}+1 & A_1(q)\times 2 \\
\hline
^2\!B_2(q) & q+\sqrt{2q}+1 & [q^2] \\
\hline
\end{array}
\]
\end{table}

\begin{table}[ht]
\caption{Maximal overgroups of $\la x\ra$, $G$ exceptional} \label{tb2}
\[
\begin{array}{|c|c|c|}
\hline
G & |{\mathcal M}(x)| & \hbox{Groups in }{\mathcal M}(x) \\
\hline
E_8(q) & 0 & - \\
\hline
E_7(q),\,q\ge 3 & 1 & ^2\!E_6(q).\frac{q+1}{d} \\
E_7(2) & 1 & ^2\!A_7(2).2 \\
\hline
E_6^\e(q) & 1 & A_2^\e(q^3).3 \\
\hline
F_4(q),\,q\ge 3 & (2,q) & ^3\!D_4(q).3 \\
F_4(2) & 2 & B_4(2) \\
\hline
G_2(q),\,q\ge 5 & 1 & SU_3(q).2 \\
G_2(4) & 2 & SU_3(4).2,\,L_2(13).2 \\
G_2(3) & 3 & SL_3(3).2,\, SL_3(3).2,\,L_2(13) \\
\hline
^3\!D_4(q),\,^2\!F_4(q) & 0 & - \\
^2\!G_2(q),\, ^2\!B_2(q) & 0 & - \\
\hline
\end{array}
\]
\end{table}

Let ${\rm Inv}(x)$ be the set of involutions in ${\rm Aut}(G)$ that invert $x$. By Lemma \ref{inv} we have 
${\rm Inv}(x) \subseteq Tt_0$. Define ${\mathcal M}(x)$ to be the set of maximal subgroups of $G$ containing $x$ such that $M \ne N_G(\la x\ra)$. It is shown in the proof of \cite[Prop. 6.2]{GK} that the set ${\mathcal M}(x)$ is as given in Table \ref{tb2}.

Define
\[
S = \bigcup_{\a \in Inv(x)} I_2(C_G(\a)) \cup \bigcup_{M\in {\mathcal M}(x)} I_2(M).
\]
We claim that 
\begin{equation}\label{i2gs}
i_2(G) > |S|.
\end{equation}
Given (\ref{i2gs}), there exists an involution $t\in G$ such that 

(a) $t$ lies in no maximal subgroup of $G$ containing $x$, and 

(b) $t$ centralizes no involution that inverts $x$.

\no It then follows that $x,t$ satisfy conditions (i)-(iii) of Theorem \ref{main1}, completing the proof of the theorem for exceptional groups.

So it remains to prove (\ref{i2gs}). By Lemma \ref{inv} we have $|{\rm Inv}(x)| \le |T|$, and also from the above discussion we know that $C_G(\a)$ is conjugate to $C_G(t_0)$ for all $\a \in {\rm Inv}(x)$. Hence 
\begin{equation}\label{upps}
|S| \le |T|\,i_2(C_G(t_0)) + \sum_{M \in {\mathcal M}(x)}i_2(M).
\end{equation}
Using Lemma \ref{upperinv} (and also the slightly stronger bound $i_2(A_1(q))< q^2$), we obtain the following upper bounds for $i_2(C_G(t_0))$ and for $\sum_{M \in {\mathcal M}(x)}i_2(M)$:
\[
\begin{array}{l|l|l}
G & i_2(C_G(t_0)) < &  \sum i_2(M)\le \\
\hline
E_8(q)& 2q^{84}(q^{20}+q^{19}) & 0 \\
E_7(q)  & 2q^{42}(q^{12}+q^{11})& 2(q+1)(q^{42}+q^{41}) \\
E_6^\e(q)& 2q^{15}(q^{12}+q^{11})& 2(q^{15}+q^{12}) \\
F_4(q) &   q^{22} & 2(q^{20}+q^{19}) \\
G_2(q),\,q\ge 5 &  q^5 & 4(q^{5}+q^{4}) \\
^3\!D_4(q) &  q^{11}& 0 \\
^2\!F_4(q),\,q>2 & q^{11} &0 \\
^2\!G_2(q),\,q>3 &  2q^2 & 0 
\end{array}
\]
In all these cases we check that the consequent upper bound for $|S|$ using (\ref{upps}) is less than the lower bound for $i_2(G)$ given by Lemma \ref{lowerinv}, proving (\ref{i2gs}). 

This leaves the following groups to deal with: $G = \,^2\!B_2(q)$, $G_2(3)$, $G_2(4)$ and $^2\!F_4(2)'$. In the first case, we see using \cite{suz} that $ i_2(C_G(t_0)) = q-1$, while $i_2(G) = (q-1)(q^2+1)$, so (\ref{i2gs}) holds. In the other cases we can use \cite{Con85} to obtain the precise values of $i_2(G)$, $i_2(C_G(t_0))$ and $\sum i_2(M)$, and again check that (\ref{i2gs}) holds.

This completes the proof of Theorem \ref{main1} when $G$ is an exceptional group of Lie type.

\section{Proof of Theorem \ref{main1} for $G$ classical}\label{class}

In this section we prove Theorem \ref{main1} in the case where $G$ is classical.

Let $G$ be a finite simple classical group over $\F_q$, and exclude $L_2(q)$, $L_3(q)$ and $U_3(q)$. So $G$ is one of the groups
\[
L_n^\e(q)\,(n\ge 4), \;PSp_{2m}(q)\,(m\ge 2), \;P\Omega_{2m+1}(q)\, (q \hbox{ odd}),\; P\Omega_{2m}^\e(q)\, (m\ge 4).
\]
 Let $V$ be the natural module for $G$. 
As in the previous section, write $G = (\bar G^F)'$, where $F$ is a Frobenius endomorphism of the corresponding simple adjoint algebraic group $\bar G$ over $\bar \F_q$. Define
\[
d = \left\{\begin{array}{l}
(n,q-\e), \hbox{ if }G=L_n^\e(q) \\
(2,q-1), \hbox{ if }G=PSp_{2m}(q), P\Omega_{2m+1}(q) \hbox{ or } P\Omega_{2m}^\e(q). \\
\end{array} \right.
\]

For convenience we handle first the following groups $G$:
\begin{equation}\label{exclude}
\begin{array}{l}
L_4^\e(2), \,L_4^\e(3),\,L_4^\e(4),\,L_5^\e(2),\,L_6^\e(2), \\
Sp_4(2),\,PSp_4(3),\,Sp_4(4),\,Sp_6(2),\,PSp_6(3),\,Sp_8(2), \\
\Omega_7(3), \,\Omega_8^\e(2), \,P\Omega_8^+(3),\,\Omega_{10}^\e(2).
\end{array}
\end{equation}
For these groups, generators $x,t$ as in Theorem \ref{main1} can be found by a search using {\sc Magma}~\cite{Magma}. So suppose from now on that $G$ is not one of the groups in (\ref{exclude}).

The proof follows along the same lines as the previous section. There is an element $x \in G$ of order given in Table \ref{tb1cla}; in all cases we take $x$ to act either irreducibly on the natural module $V$, or irreducibly on both summands of an orthogonal decomposition $V=V_i+ V_i^\perp$, where $\dim V_i = i\le 2$.
Then $T = C_{\bar G^F}(x)$ is a maximal torus of order $d|\la x \ra|$ and also $C_{Aut(G)}(T) = T$. Again, the involutions in ${\rm Aut}(G)$ that invert $x$ are conjugates of $t_0$, a preimage in $N(T)$ of $-1 = w_0$ or $w_0\tau$ (where $\tau$ is an involutory graph automorphism). The possibilities for $C_G(t_0)$ are also given in Table \ref{tb1cla}. (In the symplectic and orthogonal cases, $C_G(t_0)$ is either the given group, or the given group quotiented by the scalars $\la - I\ra $.)

\begin{table}[htb]
\caption{Torus $\la x\ra$ and centralizer $C_G(t_0)$, $G$ classical} \label{tb1cla}
\[
\begin{array}{|c|c|c|}
\hline
G & |\la x \ra | & C_G(t_0) \\
\hline
L_n(q), & \frac{1}{d}\frac{q^n-1}{q-1} & PSO^\e_n(q),\,q \hbox{ or }n\hbox{ odd} \\
n\ge 4                   &                                                  & [q^{n-1}]Sp_{n-2}(q),\,q \hbox{ and }n\hbox{ even} \\
\hline
U_n(q), & \frac{1}{d}\frac{q^n+1}{q+1}, n\hbox{ odd} & PSO_n(q) \\
            n\ge 4       &    \frac{1}{d}(q^{n-1}+1), n\hbox{ even} & PSO^\e_n(q),\,q \hbox{ odd} \\
                  &      & [q^{n-1}]Sp_{n-2}(q),\,q\hbox{ even} \\
\hline
PSp_{2m}(q), & \frac{1}{d}(q^m+1) & GL_m^\e(q).2,\,q\equiv \e\hbox{ mod }4 \\
                      m\ge 2              & & [q^{\frac{1}{2}(m^2+3m-2)}].Sp_{m-2}(q),\,q\hbox{ and }m \hbox{ even} \\
                                 & & [q^{\frac{1}{2}(m^2+m)}].Sp_{m-1}(q),\,q\hbox{ even},\,m\hbox{ odd} \\
\hline
P\Omega_{2m+1}(q), & \frac{1}{d}(q^m+1) & (O_m^\e(q)\times O_{m+1}^{\e'}(q))\cap G \\
m\ge 3,\,q\hbox{ odd} && \\
\hline
P\Omega_{2m}^+(q), & \frac{1}{d}(q^{m-1}+1)(q+1),\,m\hbox{ odd} & (O_m^\e(q))^2.2\cap G,\,q\hbox{ odd} \\
m\ge 4  &  \frac{1}{d}(q^{m-1}+1),\,m\hbox{ even} &  [q^{\frac{1}{2}(m^2+m-2)}].Sp_{m-2}(q),\,q\hbox{ and }m \hbox{ even} \\
                                 & & [q^{\frac{1}{2}(m^2-m)}].Sp_{m-1}(q),\,q\hbox{ even},\,m\hbox{ odd} \\
\hline
P\Omega_{2m}^-(q), & \frac{1}{d}(q^m+1) & (O_m^\e(q)\times O_m^{-\e}(q))\cap G,\,q\hbox{ odd} \\
m\ge 4  &   &  [q^{\frac{1}{2}(m^2+m-2)}].Sp_{m-2}(q),\,q\hbox{ and }m \hbox{ even} \\
                                 & & [q^{\frac{1}{2}(m^2-m)}].Sp_{m-1}(q),\,q\hbox{ even},\,m\hbox{ odd} \\
\hline
\end{array}
\]
\end{table}

\begin{table}[ht]
\caption{Maximal overgroups of $\la x\ra$, $G$ classical} \label{tb2cla}
\[
\begin{array}{|c|c|c|}
\hline
G & \hbox{Groups in }{\mathcal M}(x) & \hbox{Number} \\
\hline
L_n(q) & (GL_{\frac{n}{r}}(q^r).r)\cap G, & 1 \hbox{ for each }r \\
           &  r\hbox{ prime},\,r|n & \\
\hline
U_n(q) & (GU_{\frac{n}{r}}(q^r).r)\cap G, & 1 \hbox{ for each }r \\
           &  n\hbox{ odd},\,r\hbox{ prime},\,r|n & \\           
           & GU_{n-1}(q),\,n\hbox{ even} & 1 \\
\hline
PSp_{2m}(q) &  Sp_{\frac{2m}{r}}(q^r).r, & 1 \hbox{ for each }r \\
           &  r\hbox{ prime},\,r|m & \\           
           & O_{2m}^-(q),\,q\hbox{ even} & 1 \\
\hline
P\Omega_{2m+1}(q) & O_{2m}^-(q)\cap G & 1 \\
\hline
P\Omega_{2m}^+(q) & (O_{2m-2}^-(q)\times O_2^-(q))\cap G & 1 \\
                           & (GU_{m}(q).2)\cap G,\,m\hbox{ even} & 2 \\
                           & \Omega_7(q)\,(\hbox{irred.}),\,m=4 & d \\
\hline
P\Omega_{2m}^-(q) &  (O^-_{\frac{2m}{r}}(q^r).r)\cap G, & 1 \hbox{ for each }r \\
           &  r\hbox{ prime},\,r|m & \\           
           & (GU_{m}(q).2)\cap G,\,m\hbox{ odd} & 1 \\

\hline
\end{array}
\]
\end{table}

\begin{table}[ht]
\caption{Upper bounds for $i_2(C_G(t_0))$ and $\sum i_2(M)$, $G$ classical} \label{fol}
\[
\begin{array}{|c|c|c|}
\hline
G & i_2(C_G(t_0)) < &  \sum i_2(M)\le \\
\hline
L_n^\e (q),\,n\hbox{ even} & q^{n-1}i_2(Sp_{n-2}(q)) & 2(q+1)^2q^{\frac{1}{2}(n^2-n-4)} \\
L_n^\e (q),\,n\hbox{ odd} & i_2(PSO_{n}(q)) & 2d(n)\,(q^3+1)q^{\frac{1}{6}(n^2+3n)} \\
PSp_{2m} (q),\,m\hbox{ even} & q^{\frac{1}{2}(m^2+3m-2)}i_2(Sp_{m-2}(q)) & i_2(O_{2m}^-(q)+d(m)\,i_2(Sp_m(q^2).2) \\
PSp_{2m} (q),\,m\ge 3\hbox{ odd} & q^{\frac{1}{2}(m^2+m)}i_2(Sp_{m-1}(q)) & i_2(O_{2m}^-(q))+
d(m)\,i_2(Sp_{2m/r}(q^r)), \\
&& r \hbox{ largest prime divisor of }m \\
P\Omega_{2m}^\e (q),\,m\ge 6\hbox{ even} & q^{\frac{1}{2}(m^2+m-2)}i_2(Sp_{m-2}(q)) & 
2(q+1)^2q^{(m-1)^2-1}+ \\
&& 2(q+1)^2q^{\frac{1}{2}(m^2-m+4)} \\
P\Omega_{2m}^\e (q),\,m\ge 5\hbox{ odd} & q^{\frac{1}{2}(m^2-m)}i_2(Sp_{m-1}(q)) &  \hbox{ as above} \\
P\Omega_{2m+1}(q),\,m\ge 3,q\hbox{ odd} & i_2(O_m^\e(q)\times O_{m+1}^{\e'}(q)) & i_2(O_{2m}^-(q)) \\
\hline
\end{array}
\]
\end{table}

As in the previous section, define ${\mathcal M}(x)$ to be the set of maximal subgroups of $G$ containing $x$ such that $M \ne N_G(\la x\ra)$. We claim that ${\mathcal M}(x)$ is as in Table \ref{tb2cla}. When $\la x \ra$ is a Singer subgroup (i.e. the intersection with $G$ of a cyclic subgroup generated by a Singer cycle of $PSL(V)$), this follows from the main theorem of \cite{ber}. This covers the cases where $G = L_n(q)$, $U_n(q)$ with $n$ odd, $PSp_{2m}(q)$ and $P\Omega_{2m}^-(q)$. 
For the other cases we use \cite[Thm. 3.1]{BP}, which classifies subgroups of classical groups of orders divisible by numbers of the form $\Phi_d(q)$, where $d > \frac{1}{2}\dim V$. Working through the possible subgroups, we find that the only ones containing $x$ are those in Table \ref{tb2cla}. We calculate the number of groups in ${\mathcal M}(x)$ using Lemma \ref{maxconj}.

Again let ${\rm Inv}(x)$ be the set of involutions in ${\rm Aut}(G)$ that invert $x$, and define
\[
S = \bigcup_{\a \in Inv(x)} I_2(C_G(\a)) \cup \bigcup_{M\in {\mathcal M}(x)} I_2(M).
\]
We aim to show that $i_2(G) > |S|$, which will complete the proof of Theorem \ref{main1} when $G$ is classical.

As in (\ref{upps}) we have $|S| \le |T|\,i_2(C_G(t_0)) + \sum_{M \in {\mathcal M}(x)}i_2(M)$, and so it suffices to prove that 
\begin{equation}\label{suf}
i_2(G) > |T|\,i_2(C_G(t_0)) + \sum_{M \in {\mathcal M}(x)}i_2(M).
\end{equation}
Lemma \ref{upperinv} gives the upper bounds for $i_2(C_G(t_0))$ and for $\sum_{M \in {\mathcal M}(x)}i_2(M)$ in Table \ref{fol}, where $d(n)$ denotes the number of prime divisors of $n$. We can use these bounds together with Lemma \ref{upperinv} to get an upper bound for the right hand side of (\ref{suf}); and Lemma \ref{lowerinv} gives a lower bound for $i_2(G)$. In this way we check easily that (\ref{suf}) holds with the following possible exceptions (recalling that we have already excluded the groups in (\ref{exclude})):
\begin{itemize}
\item[(1)] $G = L_4^\e(q)$, $PSp_4(q)$ or $P\Omega_8^\e(q)$;
\item[(2)] $G = L_5^\e(3),\,L_6^\e(3),\,L_8^\e(2),\,Sp_6(4),\,PSp_8(3),\,Sp_8(4)$ or $\Omega_{12}^\e(2)$.
\end{itemize}
For the groups under (1), we show that (\ref{suf}) still holds, by improving the lower bound on $i_2(G)$ from Lemma \ref{lowerinv} and the upper bound on $i_2(C_G(t_0))$ in Table \ref{fol} by direct calculation in $G$, as follows:
\[
\begin{array}{|l|l|l|}
\hline
G & i_2(G) > & i_2(C_G(t_0)) < \\
\hline
L_4^\e(q) & \frac{1}{2}q^8 & 2q^4 \\
PSp_4(q) & \frac{1}{2}q^5(q-1) & 2q^3 \\
P\Omega_8^\e(q),\,q \hbox{ even} & \frac{1}{2}q^{16} & q^{11} \\
P\Omega_8^\e(q),\,q \hbox{ odd} & \frac{1}{8}q^{16} & 4q^{8} \\
\hline
\end{array}
\]
Using these improved bounds, it is straightforward to check that (\ref{suf}) again holds.

Finally, for the groups under (2) above, we again improve the bounds on $i_2(G)$ and $i_2(C_G(t_0))$ by direct calculation to show that (\ref{suf}) holds.

This completes the proof of Theorem \ref{main1} when $G$ is a classical group. 

\section{Proof of Theorem \ref{main1} for $G$ an alternating group}\label{alt}

The existence of a pair $x,t$ in $A_n$ satisfying (i)--(iii) of Theorem \ref{main1} for every $n\geq 8$ can be extracted from papers of Conder~\cite{Conder1980, Conder1981}, a more recent paper by Conder et al.~\cite{Conder2016}, and some easy {\sc Magma} computations.
Gareth Jones recently gave the following pairs $x,t$ in his plenary lecture at the conference "Symmetries and Covers of Discrete Objects" (Queenstown, New Zealand, February 2016). 

 For even $n\geq 8$, take 
\[x = (2, 3, \ldots, n)\mathrm{\; and \;}t=(1,2)(3,4)\]
And for odd $n\geq 9$, take 
\[x = (1, 2, \ldots, n)\mathrm{\; and \;}t=(1,2)(3,6)\]
Elementary arguments show that $\langle x,t\rangle = A_n$, and 
it is an easy exercise to show that $x,t$ also satisfy (iii) of Theorem~\ref{main1}.

\section{Proof of Theorem \ref{main1} for $G$ sporadic}\label{sporadic}

In this section, we show that each of the 26 sporadic simple groups has at least one pair $x,t$ satisfying (i)--(iii) of Theorem~\ref{main1}.
We shall use the fact that such pairs give abstract chiral polyhedra, as explained in the preamble to the theorem. 
When we give a pair $x,y$ of generators of $G$, the corresponding pair $x,\, t:=xy$ is the one satisfying conditions (i)--(iii) of Theorem~\ref{main1}.

\subsection{Mathieu groups}
The groups $M_i$ with $i=11$, $12$, $22$ were fully investigated in~\cite{HHL2012}. They respectively have 66, 118 and 242 non-isomorphic chiral polyhedra, hence they have that many pairs $x$, $t$ satisfying (i)--(iii) of Theorem~\ref{main1}.

The following generators of $M_{23}$ give a chiral polyhedron of type $\{11,15\}$ (the {\it type} being $\{p,q\}$ where $x,y$ have orders $p,q$):
\[x:= (1, 14, 17, 21, 10, 5, 2, 16, 18, 12, 8)(3, 6, 19, 22, 15, 9, 20, 23, 4, 7, 11),\]
\[y:= (1, 8, 6, 10, 21, 22, 19, 12, 11, 7, 4, 5, 3, 18, 9)(2, 23, 20, 16, 13)
(14, 15, 17),\]
and $t:=xy$. 
The pair $x,t$ satisfies (i)--(iii) of Theorem~\ref{main1}.

The following generators of $M_{24}$ give a chiral polyhedron of type $\{23,15\}$:
\[x:= (1,17,23,21,2,7,3,15,4,20,10,6,16,13,19,22,11,18,5,14,9,8,12),\]
\[y:=(1, 20, 2)(3, 7, 4, 17, 21, 5, 18, 24, 11, 22, 19, 9, 14, 23, 15)
(8, 13, 16, 10, 12),\]
and $t:=xy$.
The pair $x,t$ satisfies (i)--(iii) of Theorem~\ref{main1}.

\subsection{Janko groups}
The groups $J_1$ and $J_2$ were investigated in~\cite{HHL2012}. They have respectively 1056 and 888 non-isomorphic chiral polyhedra.

A {\sc Magma} search gave a chiral polyhedron of type $\{19,8\}$ for $J_3$. We do not give its generators here as these are permutations on 6516 points.

The group $J_4$ has $i_2(J_4) = 51,747,149,311$.
It also has a unique class of maximal subgroups containing elements of order 29.
Take $\sigma\in J_4$ of order 29. The normalizer $N_{J_4}(\langle\sigma\rangle) = C_{29}:C_{28}$ is maximal in $J_4$. We have $i_2(C_{29}:C_{28}) = 29$, and all the 29 involutions are conjugate in $J_4$.
There are two conjugacy classes $2A$ and $2B$ of involutions in $J_4$. Using the character table in Atlas~\cite{Con85},
we compute that the structure constant for the classes $2B,2B,29A$ is 1, and hence  involutions in $C_{29}:C_{28}$ are of type $2B$.

Now, the centralizer of an involution of type $2B$ has structure $2^{11}:(M_{22}:2)$. This subgroup has exactly 280831 involutions. Since $29 \cdot 280831 << i_2(J_4)$, there must exist at least one chiral polyhedron with automorphism group $J_4$.

\subsection{Conway groups}
A non-exhaustive computer search with {\sc Magma} gives a chiral polyhedron of type $\{23,23\}$ for $Co_{23}$,  one of type $\{14,23\}$ for $Co_2$ and one of type $\{3,60\}$ for $Co_1$.

\subsection{Fischer groups}
The group $Fi_{24}'$ has $Out=2$.
In $Fi_{24}'$, take $\sigma$ of order 29. We have $N_{Aut(Fi_{24}')}(\langle\sigma\rangle) = C_{29}:C_{28}$. This  contains 29 involutions, all in $Fi_{24}'$.
These involutions belong to class $2B$, and their centralizer $C$ in $Fi_{24}$ satisfies $i_2(C) = 5741695$. Therefore $i_2(C) \cdot 29 << i_2(Fi_{24}')$ and so $Fi_{24}'$ is the automorphism group of at least one chiral polyhedron. 

The group $Fi_{23}$ has $Out=1$, and for an element $\sigma$ of order 23, $N_{Fi_{23}}(\langle\sigma\rangle)=C_{23}:C_{11}$. This latter group does not contain any involutions. Moreover, there is obviously at least one involution that will, with $\sigma$, generate the whole of $Fi_{23}$.

Finally, a non-exhaustive computer search with {\sc Magma} gives a chiral polyhedron of type $\{11,13\}$ for $Fi_{22}$.

\subsection{The Monster and the Baby Monster}
The Monster $M$ has $Out=1$ and a unique class of maximal subgroups of order divisible by 71, namely subgroups $L_2(71)$. Moreover, $N_{L_2(71)}(C_{71}) = C_{71}:C_{35}$, a group of odd order. Therefore, no element of order 71 in $M$ is conjugate to its inverse.
Take $x$ of order 71 in $M$. The $x$ is contained in a unique subgroup $L_2(71)$ of $M$. Therefore, picking $t$ an involution of $M$ not in the $L_2(71)$ containing $\langle x\rangle$, we have $\langle t,x\rangle = M$. The pair $x,t$ satisfies (i)--(iii) of Theorem~\ref{main1}.

The Baby Monster $BM$ has $Out=1$ and a unique class of maximal subgroups containing elements of order 47. Take $x$ an element of order $47$ in $BM$. We have $N_{BM}(\langle x\rangle)=C_{47}:C_{23}$. Any involution $t$ of $BM$ will give $\langle x,t\rangle = BM$. The pair $x,t$ satisfies (i)--(iii) of Theorem~\ref{main1}.
 
 \subsection{The remaining sporadics}
 
The Thompson group $Th$ has $Out=1$,
 and for an element $x$ of order 31, $N_{Th}(\langle x\rangle)=C_{31}:C_{15}$. This latter group does not contain any involutions. Moreover, there is obviously at least one involution that will, with $x$, generate the whole of $Th$.
  
The Lyons group $Ly$ has $Out=1$ and a unique class of maximal subgroups of order divisible by 37, namely groups $C_{37}:C_{18}$. Moreover, $Ly$ has a unique conjugacy class of involutions, and these have centralizers $2\cdot A_{11}$. We have $i_2(2\cdot A_{11}) = 34650$. Therefore $i_2(2\cdot A_{11}) \cdot 37 << i_2(Ly)$ and $Ly$ is the automorphism group of at least one chiral polyhedron.

For the O'Nan group $O'N$, we refer to~\cite{CL2013} where all possible types of chiral polyhedra for $O'N$ have been determined.

 Finally, a non-exhaustive computer search with {\sc Magma} gives a chiral polyhedron of type $\{19,20\}$ for $HN$,  one of type $\{5,7\}$ for $He$, one of type $\{14,29\}$ for $Ru$, one of type $\{13,24\}$ for $Suz$, one of type $\{11,15\}$ for $McL$, and one of type $\{11,6\}$ for $HS$.

\vspace{4mm}
This completes the proof of Theorem \ref{main1}.

\section{Concluding remarks}
It would be interesting to prove similar results to Theorem~\ref{main1} and also that of Nuzhin described in the Introduction, for almost simple (rather than just simple) groups.
Some results  in this vein are known. For instance in~\cite{HL2014}, it is proved that every almost simple group $G$ with socle $Sz(q)$ is the automorphism group of at least one abstract chiral polyhedron.
And in~\cite{LM2015}, it is shown that the only almost simple groups with socle $L_2(q)$ that are not automorphism groups of abstract chiral polyhedra are $L_2(q)$, $PGL(2,q)$, and a group of the form $L_2(9).2$.

\section{Acknowledgements}
This research was supported by a Marsden grant (UOA1218) of the Royal Society of New Zealand.

\bibliographystyle{plain}

\end{document}